\documentclass[11pt,reqno]{amsart}
\usepackage[T1]{fontenc}
\usepackage{lmodern}
\usepackage{mathtools,amssymb,enumitem,aliascnt}
\usepackage[margin=1.12in]{geometry}
\usepackage{microtype}
\usepackage[colorlinks=true,linkcolor=blue,citecolor=blue,urlcolor=blue]{hyperref}
\usepackage[capitalise,noabbrev]{cleveref}
\hypersetup{pdftitle={Uniform Hilbert embeddings of Banach spheres and Property (H)},pdfauthor={Qingjin Cheng and Yue Wang}}

\newaliascnt{proposition}{theorem}
\newtheorem{proposition}[proposition]{Proposition}
\aliascntresetthe{proposition}
\crefname{proposition}{proposition}{propositions}
\Crefname{proposition}{Proposition}{Propositions}
\newaliascnt{lemma}{theorem}
\newtheorem{lemma}[lemma]{Lemma}
\aliascntresetthe{lemma}
\crefname{lemma}{lemma}{lemmas}
\Crefname{lemma}{Lemma}{Lemmas}
\newaliascnt{corollary}{theorem}

\aliascntresetthe{corollary}
\crefname{corollary}{corollary}{corollaries}
\Crefname{corollary}{Corollary}{Corollaries}
\newtheorem{theoremA}{Theorem}

\Crefname{theoremA}{Theorem}{Theorems}
\theoremstyle{definition}
\newaliascnt{definition}{theorem}
\newtheorem{definition}[definition]{Definition}
\aliascntresetthe{definition}
\crefname{definition}{definition}{definitions}
\Crefname{definition}{Definition}{Definitions}
\theoremstyle{remark}
\newaliascnt{remark}{theorem}

\aliascntresetthe{remark}
\crefname{remark}{remark}{remarks}
\Crefname{remark}{Remark}{Remarks}
\numberwithin{equation}{section}
\DeclareMathOperator{\dist}{dist}
\DeclareMathOperator{\ran}{ran}

\DeclareMathOperator{\Lip}{Lip}
\DeclarePairedDelimiter{\norm}{\lVert}{\rVert}
\DeclarePairedDelimiter{\abs}{\lvert}{\rvert}
\DeclarePairedDelimiter{\ip}{\langle}{\rangle}
\newcommand{\R}{\mathbb R}

\newcommand{\Id}{I}

\setlist{itemsep=3pt,topsep=5pt}

\title[Banach spheres and Property~(H)]{Uniform Embeddings of Banach Space Spheres into Hilbert Spaces and Property (H)}
\author{Qingjin Cheng}
\address{School of Mathematical Sciences, Xiamen University, Xiamen 361005, China}
\email{qjcheng@xmu.edu.cn}

\author{Yue Wang}
\address{School of Mathematical Sciences, Xiamen University, Xiamen 361005, China}
\email{19020260158074@stu.xmu.edu.cn}

\author{Yipeng Zhang}
\address{School of Mathematical Sciences, Xiamen University, Xiamen 361005, China}
\email{19020250157654@stu.xmu.edu.cn}
\date{}
\begin{document}
\begin{abstract}
We prove that if the unit sphere of a separable infinite-dimensional real Banach space  embeds uniformly into a Hilbert space, then the Banach space has Kasparov--Yu's Property~(H). 
\end{abstract}
\maketitle

\section{Introduction and main theorem}\label{sec:introduction}

Kasparov and Yu introduced Property~(H) and proved the Strong Novikov conjecture with coefficients for countable discrete groups that coarsely embed into Banach spaces with this property \cite{KY}. Their condition combines finite-dimensional approximation with topological degree: along a paving of the Banach space, one requires compatible sphere maps into a Hilbert space, of degree one on every finite-dimensional sphere. The purpose of this paper is to establish that uniform embeddability of the unit sphere of a Banach space into a Hilbert space implies Property~(H).

Unless otherwise stated, Banach spaces are real. For a normed space $E$, we denote its closed unit ball and unit sphere by
\[
B_E=\{x\in E:\norm x\le1\}, \qquad S_E=\{x\in E:\norm x=1\},
\]
and equip them with the distance induced by the norm. We write $E^*$ for the continuous dual of $E$. Hilbert norms are denoted by $\norm{\cdot}_2$, and inner products by $\ip{\cdot,\cdot}$. A \emph{paving} of a separable infinite-dimensional Banach space $X$ is a strictly increasing sequence $(V_n)_{n\ge1}$ of finite-dimensional subspaces with dense union. We assume $\dim V_1\ge2$; every increasing sequence with dense union has a subsequence with these properties. Finite-dimensional spaces are oriented, their unit spheres carry the boundary orientations, and products carry the product orientations. The degree of a continuous map between connected closed oriented manifolds of the same dimension is the integer by which the map acts on top singular homology with integer coefficients.

A \emph{modulus of continuity} is a nondecreasing function $\omega:[0,2]\to[0,2]$ such that $\omega(0)=0$ and $\lim_{t\to0}\omega(t)=0$. For maps into unit spheres, upper bounds are truncated at $2$ when regarded as moduli of continuity.

\begin{definition}[Kasparov--Yu \cite{KY}]\label{def:KY}
A separable infinite-dimensional Banach space $X$ has \emph{Property~(H)} if there are pavings $(V_n)$ of $X$ and $(H_n)$ of a real Hilbert space $H$, with $\dim H_n=\dim V_n$, and a uniformly continuous map
\[
\psi:S_V\longrightarrow S_W,\qquad V=\bigcup_n V_n,\quad W=\bigcup_n H_n,
\]
such that
\[
\psi(S_{V_n})\subset S_{H_n}, \qquad \deg(\psi|_{S_{V_n}})=1 \qquad(n\ge1).
\]
Replacing degree one by nonzero degree defines \emph{rational Property~(H)}.
\end{definition}

We use this degree formulation throughout. It is explicitly allowed in \cite[Definition~1.1]{KY}; no homeomorphism condition is imposed. The classical Mazur maps give Property~(H) for $\ell_p$, $1\le p<\infty$, and the noncommutative Mazur maps give it for the Schatten classes \cite[pp.~1859--1860]{KY}. Cheng and Wang \cite{CW} proved Property~(H) for all separable infinite-dimensional Banach lattices of finite cotype.

We shall use the standard notion of a uniform embedding. A map $f:\mathcal M\to\mathcal N$ between metric spaces is a \emph{uniform embedding} if $f$ is uniformly continuous and $f^{-1}:f(\mathcal M)\to\mathcal M$ is uniformly continuous. Equivalently, with
\[
\omega_f(t)=\sup\{d_{\mathcal N}(f(x),f(y)):d_{\mathcal M}(x,y)\le t\}
\]
and
\[
\rho_f(t)=\inf\{d_{\mathcal N}(f(x),f(y)):d_{\mathcal M}(x,y)\ge t\},
\]
one has $\lim_{t\to0}\omega_f(t)=0$ and $\rho_f(t)>0$ for every $t>0$.

Our main result is the following.

\begin{theoremA}\label{thm:A}
Let $X$ be a separable infinite-dimensional real Banach space. If $S_X$ admits a uniform embedding into a Hilbert space, then $X$ has Property~(H).
\end{theoremA}

We first compress the uniform embedding to finite-dimensional Hilbert balls. A fixed-scale lower estimate allows us to recover degree after taking products with Euclidean spheres. Radius removal, a join construction, and Dvoretzky's theorem then produce a paving $(V_n)$ and degree-one maps
\[
f_n:S_{V_n}\longrightarrow S_{\ell_2^{\dim V_n}}
\]
with one common modulus of continuity. No compatibility between the maps $f_n$ is imposed at this stage.

The maps $f_n$ need not agree on the smaller spheres. We next replace them by maps with a common H\"older estimate and extend them to $S_X$ by Minty's theorem. One additional Hilbert coordinate keeps the extensions nonzero. We then localize these extensions on pairwise disjoint caps and paste them together. On each $S_{V_n}$, the last cap has the prescribed degree, while the earlier caps have degree zero. The resulting map satisfies the conditions in Property~(H).

Section~\ref{sec:finite} constructs the finite-dimensional degree-one maps. Section~\ref{sec:compatibility} pastes them into one uniformly continuous map and completes the proof of Theorem~A. Section~\ref{sec:applications} gives applications to noncommutative $L_p$-spaces and to the Novikov conjecture.

\section{Finite-dimensional degree maps from a Hilbert embedding}\label{sec:finite}

Throughout this section, $X$ is a separable infinite-dimensional real Banach space. For $r>0$, let
\[
B_E(r)=rB_E, \qquad S_E(r)=rS_E.
\]
For a nonempty subset $A$ of a normed space, let
\[
\dist(x,A)=\inf_{a\in A}\norm{x-a}.
\]
We write $\Lip(f)$ for the least Lipschitz constant of $f$, and use the scalar notation
\[
t_+=\max\{t,0\}.
\]
For a smooth map $f$, the symbols $Df$ and $D^2f$ denote its first and second Fr\'echet derivatives. Norms of linear and bilinear maps are operator norms. For a bounded linear map $T$ between Hilbert spaces, $T^*$ denotes its adjoint, and $\ran T$ denotes its range. The symbol $\Id$ denotes the identity operator on the space in question.

We first construct degree-one maps on a paving of $X$ with one common modulus of continuity. Section~\ref{sec:compatibility} will then paste these maps into a single map on $S_X$.

\begin{proposition}\label{prop:finite-paving}
Let $X$ be a separable infinite-dimensional real Banach space. If $S_X$ admits a uniform embedding into a Hilbert space, then there are a paving $(V_n)$ of $X$ and degree-one maps
\[
f_n:S_{V_n}\longrightarrow S_{\ell_2^{\dim V_n}}
\]
with a common modulus of continuity. If the given embedding is $\alpha$-H\"older, $0<\alpha\le1$, the maps $f_n$ may be chosen $\alpha$-H\"older with a common constant depending only on the embedding.
\end{proposition}

We first construct the required maps after a finite-dimensional enlargement of an arbitrary subspace of $X$.

\subsection{Finite-dimensional Hilbert approximation}
\label{subsec:hilbertian}

Orthogonal projections onto finite-dimensional Hilbert subspaces preserve the modulus of continuity. Uniform continuity of the inverse will be used only at one fixed scale.

A Hilbert space uniformly embeds into a Hilbert sphere; see \cite[Corollary~8.11]{BL2000}. For a real Hilbert space $H$, consider the map
\[
J(h)=e^{-\norm{h}_2^2/2} \left( 1,h,\frac{h^{\otimes2}}{\sqrt{2!}},\ldots \right), \qquad h\in H.
\]
It takes values in the unit sphere of the Hilbert direct sum
\[
\left(\bigoplus_{n=0}^{\infty}H^{\otimes n}\right)_{\ell_2},
\]
where $H^{\otimes0}=\R$ and $H^{\otimes n}$ is the $n$-fold Hilbert tensor product. Here $h^{\otimes n}$ denotes the tensor product of $n$ copies of $h$. Moreover,
\[
\norm{J(h)-J(k)}_2^2 = 2\left(1-e^{-\norm{h-k}_2^2/2}\right).
\]
Hence $J$ is Lipschitz and $J^{-1}$ is uniformly continuous on $J(H)$. We may therefore assume that the given uniform embedding of $S_X$ takes values in a Hilbert sphere. On every finite-dimensional sphere $S_E$, its image is compact and can be approximated by a finite-dimensional Hilbert subspace.

\begin{lemma}\label{lem:embedding-data}
Let $\Phi:S_X\to S_{\ell_2}$ be a uniform embedding. Then there are a modulus $\omega$ and a constant $D\ge1$ such that, for every finite-dimensional subspace $E\subset X$ with $\dim E\ge2$, there are an integer $m_E\ge1$ and a map
\[
\varphi_E:S_E\longrightarrow B_{\ell_2^{m_E}}
\]
satisfying
\[
\norm{\varphi_E(x)-\varphi_E(y)}_2 \le \omega(\norm{x-y})
\]
and
\[
\norm{x-y} \le D\norm{\varphi_E(x)-\varphi_E(y)}_2+\frac1{32}
\]
for all $x,y\in S_E$.
\end{lemma}

\begin{proof}
Let $\omega$ be a modulus for $\Phi$. Since $\Phi^{-1}$ is uniformly continuous on $\Phi(S_X)$, choose $a\in(0,1]$ so that $\norm{\Phi(x)-\Phi(y)}_2\le a$ implies $\norm{x-y}\le1/64$, and set $D=2/a$. Then
\[
\norm{x-y} \le D\norm{\Phi(x)-\Phi(y)}_2+\frac1{64} \qquad(x,y\in S_X).
\]
Fix a finite-dimensional subspace $E\subset X$. Since $\Phi(S_E)$ is compact in $\ell_2$, there is an integer $m_E\ge1$ such that the coordinate projection $P_{m_E}:\ell_2\to\ell_2^{m_E}$ satisfies
\[
\sup_{x\in S_E} \norm{\Phi(x)-P_{m_E}\Phi(x)}_2 \le\frac1{128D}.
\]
Set
\[
\varphi_E=P_{m_E}\Phi|_{S_E}.
\]
Since $P_{m_E}$ is a contraction,
\[
\norm{\varphi_E(x)-\varphi_E(y)}_2 \le \omega(\norm{x-y}).
\]
Moreover,
\[
\norm{\Phi(x)-\Phi(y)}_2 \le \norm{\varphi_E(x)-\varphi_E(y)}_2+\frac1{64D},
\]
and therefore
\[
\norm{x-y} \le D\norm{\varphi_E(x)-\varphi_E(y)}_2+\frac1{32}.
\]
\end{proof}
We use these estimates to construct a degree-one sphere map after enlarging $E$ inside $X$.

\subsection{Degree recovery on product manifolds}
\label{subsec:product}

Fix a finite-dimensional subspace $E\subset X$ with $d=\dim E\ge2$. By \Cref{lem:embedding-data}, there are an integer $m\ge1$, a map $\varphi:S_E\to B_{\ell_2^m}$, a modulus $\omega$, and a constant $D\ge1$, independent of $E$, such that
\begin{align}
\norm{\varphi(x)-\varphi(y)}_2 &\le \omega(\norm{x-y}), \label{eq:criterion-upper}\\
\norm{x-y} &\le D\norm{\varphi(x)-\varphi(y)}_2+\frac18 \label{eq:criterion-lower}
\end{align}
for all $x,y\in S_E$. Here we have only weakened the constant $1/32$ in \Cref{lem:embedding-data} to $1/8$.

The lower estimate \eqref{eq:criterion-lower} gives a map back to $S_E$ from a neighborhood of $\varphi(S_E)$. This map is used to compute the degree of a map on the product with a Euclidean sphere.

We shall use two elementary estimates. For nonzero vectors $a,b$ in a normed space,
\begin{equation}\label{eq:radial}
\left\lVert \frac{a}{\norm a}-\frac{b}{\norm b} \right\rVert \le \frac{2\norm{a-b}}{\max\{\norm a,\norm b\}}.
\end{equation}
For a uniformly continuous map $f:S_E\to Y$ into a normed space, put
\[
\omega_f(s)=\sup\{\norm{f(x)-f(y)}:\norm{x-y}\le s,\ x,y\in S_E\}.
\]
Sch\"affer's intrinsic-metric estimate \cite[Theorem~3.5]{Schaffer1967} gives
\begin{equation}\label{eq:sphere-chain}
\omega_f(s)\le(2s/t+1)\omega_f(t) \qquad(0\le s\le2,\ 0<t\le2).
\end{equation}
Indeed, two points of $S_E$ at distance $s$ can be joined in $S_E$ by an arc of length at most $2s$, and the arc can be divided into pieces of length at most $t$.

The support of a map is the closure of the set where it is nonzero.

\begin{lemma}\label{lem:local-return}
Let $E$ be a $d$-dimensional normed space, $d\ge2$, and let $\varphi:S_E\to B_{\ell_2^m}$ be continuous and satisfy
\[
\norm{x-y} \le D\norm{\varphi(x)-\varphi(y)}_2+\frac18 \qquad (x,y\in S_E)
\]
for some $D\ge1$. Put $\tau=(16D)^{-1}$ and $K=\varphi(S_E)$. Then there is a Lipschitz map $\beta:\R^m\to B_{\ell_2^d}$ such that $\norm{\beta(z)}_2=1$ when $\dist(z,K)\le\tau$ and $\beta(z)=0$ when $\dist(z,K)\ge3\tau/2$, with $\operatorname{supp}\beta\subset\operatorname{int}B_{\ell_2^m}(2)$. Moreover, $\beta\circ\varphi:S_E\to S_{\ell_2^d}$ has degree one.
\end{lemma}

\begin{proof}
Choose $x_1,\ldots,x_N\in S_E$ so that $\{\varphi(x_i)\}_{i=1}^N$ is a $\tau/4$-net in $K$. For $\dist(z,K)\le3\tau/2$, set
\begin{equation}\label{eq:local-return}
a(z)= \frac{\sum_{i=1}^N (2\tau-\norm{z-\varphi(x_i)}_2)_+x_i} {\sum_{i=1}^N (2\tau-\norm{z-\varphi(x_i)}_2)_+}.
\end{equation}
The denominator is at least $\tau/4$, so $a$ is Lipschitz.

Let $\varphi(x)\in K$ be nearest to $z$. If the $i$th term in \eqref{eq:local-return} is nonzero, then $\norm{\varphi(x_i)-\varphi(x)}_2<7\tau/2$. Hence $\norm{x_i-x}<11/32$. Since $a(z)$ is a convex combination of such $x_i$, we have $\norm{a(z)-x}<11/32$, and therefore $\norm{a(z)}>1/2$. When $z=\varphi(x)$, the same argument gives $\norm{a(\varphi(x))-x}\le1/4$.

Choose an orientation-preserving linear isomorphism $T:E\to\R^d$. Since $a$ does not vanish, set
\[
b(z)=Ta(z)/\norm{Ta(z)}_2
\]
on the $3\tau/2$-neighborhood of $K$. Choose a Lipschitz function $\chi:[0,\infty)\to[0,1]$ with $\chi=1$ on $[0,\tau]$ and $\chi=0$ on $[3\tau/2,\infty)$, and define
\[
\beta(z)=\chi(\dist(z,K))b(z)
\]
on this neighborhood and $\beta(z)=0$ outside it. Since $K\subset B_{\ell_2^m}$ and $3\tau/2\le3/32$, the support of $\beta$ is contained in $\operatorname{int}B_{\ell_2^m}(2)$.

Finally, the segment joining $x$ to $a(\varphi(x))$ avoids the origin, because $\norm{a(\varphi(x))-x}\le1/4$. Applying $T$ and normalizing gives a homotopy from $\beta\circ\varphi$ to $x\mapsto Tx/\norm{Tx}_2$. The latter has degree one, and hence so does $\beta\circ\varphi$.
\end{proof}
Combining $\beta$ and $\varphi$ gives a map of absolute degree one on a product of spheres.

\begin{lemma}\label{lem:product}
Let $E$ be a $d$-dimensional normed space, $d\ge2$, and let $\varphi:S_E\to B_{\ell_2^m}$ satisfy
\[
\norm{\varphi(x)-\varphi(y)}_2 \le \omega(\norm{x-y})
\]
and
\[
\norm{x-y} \le D\norm{\varphi(x)-\varphi(y)}_2+\frac18 \qquad (x,y\in S_E),
\]
for some modulus $\omega$ and some $D\ge1$. Then there are $R\ge1$ and a continuous map
\[
F:S_E\times S_{\ell_2^{m+1}}(R) \longrightarrow S_{\ell_2^{d+m}}
\]
such that $|\deg F|=1$ and
\begin{align}
\norm{F(x,z)-F(y,z)}_2 &\le 32D\,\omega(\norm{x-y}), \label{eq:F-x}\\
\norm{F(x,z)-F(x,w)}_2 &\le \norm{z-w}_2. \label{eq:F-z}
\end{align}
The radius $R$ may depend on $E$.
\end{lemma}
\begin{proof}
Let $\tau=(16D)^{-1}$, $K=\varphi(S_E)$, and let $\beta:\R^m\to B_{\ell_2^d}$ be given by \Cref{lem:local-return}. Identify the upper and lower hemispheres of $S_{\ell_2^{m+1}}$ with two copies of $B_{\ell_2^m}(3)$ by $(s',s_{m+1})\mapsto3s'$. Set
\[
\theta(z)=\min\{1,3-\norm z_2\}
\]
and fix $e_0\in S_{\ell_2^d}$. On the two hemispheres define
\begin{align}
F_+(x,z) &= \frac{(\theta(z)\varphi(x)-z,\beta(z))} {\norm{(\theta(z)\varphi(x)-z,\beta(z))}_2}, \label{eq:Fplus}\\
F_-(x,z) &= \left( -\frac z3, \sqrt{1-\frac{\norm z_2^2}{9}}\,e_0 \right). \label{eq:Fminus}
\end{align}

The denominator in \eqref{eq:Fplus} is bounded below by $\tau$. Indeed, if $\dist(z,K)\le\tau$, then $\norm{\beta(z)}_2=1$. If $\dist(z,K)>\tau$ and $\norm z_2\le2$, then $\theta(z)=1$ and $\norm{\varphi(x)-z}_2>\tau$. Finally, if $\norm z_2\ge2$, then $\norm{\theta(z)\varphi(x)-z}_2\ge1$.

For $\norm z_2=3$, both formulas equal $(-z/3,0)$, so they define a continuous map
\[
F_0:S_E\times S_{\ell_2^{m+1}} \longrightarrow S_{\ell_2^{d+m}}.
\]
Only the first component of the numerator in \eqref{eq:Fplus} depends on $x$. Hence \eqref{eq:radial} and \eqref{eq:criterion-upper} give
\[
\norm{F_0(x,z)-F_0(y,z)}_2 \le 32D\,\omega(\norm{x-y}).
\]
The lower hemisphere is independent of $x$, so the same estimate holds everywhere.

For fixed $x$, the upper-hemisphere formula is Lipschitz in $z$, with a constant independent of $x$, because $\theta$ and $\beta$ are Lipschitz, $\norm{\varphi(x)}_2\le1$, and the denominator is at least $\tau$. In the original sphere coordinate $s=(s',s_{m+1})$, the lower-hemisphere formula is $(-s',-s_{m+1}e_0)$ and is $1$-Lipschitz. The two formulas agree on the equator. Joining points in different hemispheres by a shortest spherical arc therefore gives a common Lipschitz constant $L$ in the second variable. Choose $R\ge\max\{1,L\}$ and define $F(x,z)=F_0(x,z/R)$ for $z\in S_{\ell_2^{m+1}}(R)$. Then \eqref{eq:F-x} is unchanged and \eqref{eq:F-z} holds.

To compute the degree, let
\[
b=\beta\circ\varphi:S_E\to S_{\ell_2^d}.
\]
The proof of \Cref{lem:local-return} gives a homotopy from $b$ to the orientation-preserving radial homeomorphism $r_T(x)=Tx/\norm{Tx}_2$, where $T:E\to\R^d$ is an orientation-preserving linear isomorphism. Choose a continuous function $\lambda:B_{\ell_2^m}(3)\to[0,1]$ such that $\lambda(z)=1$ for $\norm z_2\le1$ and $\lambda(z)=0$ for $\norm z_2\ge2$.

Keep the lower-hemisphere map fixed. On the upper hemisphere, normalize the vectors
\[
\bigl(\theta(z)\varphi(x)-z,\ (1-t)\beta(z)+t\lambda(z)b(x)\bigr), \qquad 0\le t\le1.
\]
These vectors never vanish. Indeed, if the first component is zero, then $\norm z_2\le1$, so $\theta(z)=\lambda(z)=1$ and $z=\varphi(x)$. The second component is then $b(x)$, which has norm one. On the equator $\norm z_2=3$, the normalized vector is always $(-z/3,0)$. Thus this homotopy extends over the whole product.

Next, on the upper hemisphere normalize
\[
\bigl((1-t)\theta(z)\varphi(x)-z,\ \lambda(z)b(x)\bigr), \qquad 0\le t\le1.
\]
If the first component vanishes, then again $\norm z_2\le1$, and the second component has norm one. This homotopy also fixes the equator. Its endpoint is
\[
\frac{(-z,\lambda(z)b(x))} {\norm{(-z,\lambda(z)b(x))}_2}.
\]
Finally, substitute the homotopy from $b$ to $r_T$ into this formula. The denominator remains nonzero: if $z=0$, the second component is a unit vector. We obtain a map $\widetilde F$ homotopic to $F_0$, whose upper-hemisphere formula is
\[
\widetilde F_+(x,z) = \frac{(-z,\lambda(z)r_T(x))} {\norm{(-z,\lambda(z)r_T(x))}_2},
\]
and whose lower-hemisphere formula is still \eqref{eq:Fminus}.

Choose $y\in S_{\ell_2^d}\setminus\{e_0\}$. The point $(0,y)$ has no preimage in the lower hemisphere and has exactly one preimage in the upper hemisphere, namely $(r_T^{-1}(y),0)$. Near this point, $\lambda(z)=1$ and
\[
\widetilde F_+(x,z) = \frac{(-z,r_T(x))}{\sqrt{1+\norm z_2^2}}.
\]
This is a local homeomorphism. Indeed, for a nearby point $(a,c)\in\R^m\oplus\R^d$ on the unit sphere, its inverse is $z=-a/\norm c_2$ and $x=r_T^{-1}(c/\norm c_2)$. The unique preimage therefore has local degree $\pm1$, and $|\deg\widetilde F|=1$. Homotopy invariance gives $|\deg F_0|=1$, and rescaling the second sphere gives $|\deg F|=1$. The homotopies are used only to compute the degree; the map $F$ satisfying \eqref{eq:F-x}--\eqref{eq:F-z} is unchanged.
\end{proof}

The radius $R$ in \Cref{lem:product} may depend on $E$. To replace this sphere by a unit sphere while preserving the estimates, we use a Gaussian map with dimension-independent derivative bounds.

\begin{lemma}\label{lem:gaussian}
For every $k\ge1$ there are a real Hilbert space $\mathcal H$ and a smooth map $\phi:\R^k\to S_{\mathcal H}$ such that
\begin{equation}\label{eq:gaussian}
\ip{\phi(z),\phi(w)} =e^{-\norm{z-w}_2^2/2}, \qquad \norm{D\phi(z)v}_2=\norm v_2,
\end{equation}
and
\begin{equation}\label{eq:gaussian-second}
\norm{D^2\phi(z)(v,w)}_2^2 = \norm v_2^2\norm w_2^2 +2\ip{v,w}^2 \le3\norm v_2^2\norm w_2^2.
\end{equation}
Moreover, if $K\subset\R^k$ is compact and $\varepsilon>0$, there is a finite-rank orthogonal projection $P$ on $\mathcal H$ such that
\[
\sup_{z\in K}\norm{(\Id-P)\phi(z)}_2<\varepsilon \quad\text{and}\quad \sup_{\substack{z\in K\\ \norm v_2\le1}} \norm{(\Id-P)D\phi(z)v}_2<\varepsilon.
\]
\end{lemma}

\begin{proof}
Let $\gamma_k$ be the standard Gaussian measure on $\R^k$ and set
\[
\mathcal H=L_2(\gamma_k;\R^2).
\]
Define
\[
\phi(z)(u) = \bigl(\cos\ip{u,z},\sin\ip{u,z}\bigr).
\]
Then $\norm{\phi(z)}_2=1$, and the Gaussian characteristic function gives
\[
\ip{\phi(z),\phi(w)} = \int_{\R^k}\cos\ip{u,z-w}\,d\gamma_k(u) = e^{-\norm{z-w}_2^2/2}.
\]
Differentiation gives
\[
D\phi(z)v(u) = \ip{u,v} \bigl(-\sin\ip{u,z},\cos\ip{u,z}\bigr)
\]
and
\[
D^2\phi(z)(v,w)(u) = -\ip{u,v}\ip{u,w}\phi(z)(u).
\]
The Gaussian second- and fourth-moment identities now give \eqref{eq:gaussian} and \eqref{eq:gaussian-second}. The same estimates justify the differentiations by dominated convergence.

For compact $K$, the sets $\{\phi(z):z\in K\}$ and $\{D\phi(z)v:z\in K,\ \norm v_2\le1\}$ are compact in $\mathcal H$. Choose a finite $\varepsilon$-net of their union and let $P$ be the orthogonal projection onto its linear span. Since orthogonal projection gives the best approximation from its range, the two required estimates follow.
\end{proof}

The Gaussian map allows us to place $S_{\ell_2^{m+1}}(R)$ inside a finite-dimensional Euclidean sphere with a normal neighborhood whose width and inverse estimate are independent of $m$ and $R$.
\begin{lemma}\label{lem:tube}
There are absolute constants $r\in(0,1)$ and $D_0<\infty$ such that, for every $m\ge1$ and $R\ge1$, there are an integer $q\ge1$ and a smooth embedding
\[
\Theta: S_{\ell_2^{m+1}}(R)\times B_{\ell_2^q}(2r) \longrightarrow S_{\ell_2^{m+q+1}}
\]
with the following properties.

\begin{enumerate}[label=\textnormal{(\roman*)}]
\item
For all $z,w\in S_{\ell_2^{m+1}}(R)$ and $h,k\in B_{\ell_2^q}(2r)$,
\begin{equation}\label{eq:tube-inverse}
\norm{h-k}_2+\min\{\norm{z-w}_2,1\} \le D_0\norm{\Theta(z,h)-\Theta(w,k)}_2.
\end{equation}

\item
The set
\[
\Theta\bigl( S_{\ell_2^{m+1}}(R)\times \operatorname{int}B_{\ell_2^q}(2r) \bigr)
\]
is open in $S_{\ell_2^{m+q+1}}$, and
\[
\partial\Theta\bigl( S_{\ell_2^{m+1}}(R)\times B_{\ell_2^q}(r) \bigr) = \Theta\bigl( S_{\ell_2^{m+1}}(R)\times\partial B_{\ell_2^q}(r) \bigr),
\]
where the boundary is taken in $S_{\ell_2^{m+q+1}}$.

\item
For every $z\in S_{\ell_2^{m+1}}(R)$, the vectors $\partial\Theta/\partial h_j(z,0)$, $1\le j\le q$, form an orthonormal basis of the normal space in $S_{\ell_2^{m+q+1}}$ to $\Theta(S_{\ell_2^{m+1}}(R)\times\{0\})$ at $\Theta(z,0)$, and depend smoothly on $z$.
\end{enumerate}
\end{lemma}

\begin{proof}
Apply \Cref{lem:gaussian} with $k=m+1$. Let $\eta>0$ be chosen later. The finite-rank approximation in that lemma gives an orthogonal projection $P:\mathcal H\to\mathcal H$ of finite rank such that
\begin{equation}\label{eq:projection-C1}
\begin{split}
\sup_{\norm{z}_2\le R+1} \norm{(\Id-P)\phi(z)}_2&\le\eta,\\
\sup_{\substack{\norm{z}_2\le R+1\\ \norm{v}_2\le1}} \norm{(\Id-P)D\phi(z)v}_2&\le\eta.
\end{split}
\end{equation}
Let $H_0=\ran P$, put $q=\dim H_0$, and define
\[
\Phi(z)=\frac{P\phi(z)}{\norm{P\phi(z)}_2}, \qquad z\in B_{\ell_2^{m+1}}(R+1).
\]
For $\eta$ sufficiently small, $\norm{P\phi(z)}_2\ge1-\eta$, and $\Phi$ is smooth on a neighborhood of this ball. Moreover,
\begin{equation}\label{eq:Phi-bounds}
\norm{\Phi(z)-\phi(z)}_2\le2\eta,\qquad \frac12\norm{v}_2\le\norm{D\Phi(z)v}_2\le2\norm{v}_2,\qquad \norm{D^2\Phi(z)}\le B,
\end{equation}
where $B$ is absolute. Indeed, the first two derivatives of $u\mapsto u/\norm{u}_2$ are uniformly bounded when $\norm{u}_2\ge1/2$. The first two estimates follow from \eqref{eq:projection-C1}; for the last one use $D^2(P\phi)=PD^2\phi$ and \eqref{eq:gaussian-second}.

Choose an absolute $s_0\in(0,1)$ sufficiently small. Taylor's formula and \eqref{eq:Phi-bounds} give
\begin{equation}\label{eq:Phi-local}
\norm{\Phi(z)-\Phi(w)}_2 \ge\frac14\norm{z-w}_2 \qquad \bigl(\norm{z-w}_2\le s_0\bigr)
\end{equation}
for $z,w\in S_{\ell_2^{m+1}}(R)$. If $\norm{z-w}_2\ge s_0$, then \eqref{eq:gaussian} yields
\[
\norm{\phi(z)-\phi(w)}_2^2 = 2\bigl(1-e^{-\norm{z-w}_2^2/2}\bigr) \ge 2\bigl(1-e^{-s_0^2/2}\bigr).
\]
Taking $\eta$ sufficiently small in terms of $s_0$, we obtain an absolute constant $b_0>0$ such that
\begin{equation}\label{eq:Phi-far}
\norm{\Phi(z)-\Phi(w)}_2\ge b_0 \qquad \bigl(\norm{z-w}_2\ge s_0\bigr).
\end{equation}
The estimates in \eqref{eq:Phi-bounds} remain uniform as $\eta$ decreases, so any further smallness condition on $s_0$ depending only on those bounds may be imposed before choosing $\eta$.

We construct a smooth orthonormal frame for the normal spaces. For $z\in S_{\ell_2^{m+1}}(R)$, let
\[
u=\frac zR, \qquad P_u=\Id-u\otimes u,
\]
where $(u\otimes u)a=\ip{a,u}u$ for $a\in\R^{m+1}$. Thus $P_u$ is the orthogonal projection onto $u^\perp$. Define $T_z:\R^{m+1}\to H_0$ by
\begin{equation}\label{eq:Tz}
T_za=\ip{a,u}\Phi(z)+D\Phi(z)P_ua.
\end{equation}
Since $D\Phi(z)(u^\perp)\perp\Phi(z)$,
\[
\norm{T_za}_2^2 = \abs{\ip{a,u}}^2+\norm{D\Phi(z)P_ua}_2^2.
\]
Hence \eqref{eq:Phi-bounds} gives
\[
\norm a_2^2/4\le\norm{T_za}_2^2\le4\norm a_2^2.
\]
Thus the spectrum of $T_z^*T_z$ lies in $[1/4,4]$, and $(T_z^*T_z)u=u$. Set
\begin{equation}\label{eq:Vz}
V_z=T_z(T_z^*T_z)^{-1/2}.
\end{equation}
Then $V_z:\R^{m+1}\to H_0$ is an isometry, $V_zu=\Phi(z)$, and
\[
\ran V_z = \operatorname{span}\{\Phi(z),D\Phi(z)(u^\perp)\}.
\]
Since $R\ge1$, \eqref{eq:Phi-bounds} gives
\[
\norm{T_z-T_w}\le C\norm{z-w}_2
\]
for an absolute constant $C$. For positive operators $A,B$ with spectra in $[1/4,4]$, the formula
\[
A^{-1/2} = \frac1\pi \int_0^\infty t^{-1/2}(A+t\Id)^{-1}\,dt
\]
and the resolvent identity
\[
(A+t\Id)^{-1}-(B+t\Id)^{-1} = (A+t\Id)^{-1}(B-A)(B+t\Id)^{-1}
\]
give
\[
\norm{A^{-1/2}-B^{-1/2}}\le C\norm{A-B},
\]
with an absolute constant $C$. Consequently,
\[
\norm{V_z-V_w}\le C\norm{z-w}_2.
\]
The dependence of $V_z$ on $z$ is smooth as well.

Embed $S_{\ell_2^{m+1}}(R)$ into $S_{H_0\oplus\R^{m+1}}$ by $i(z)=(\Phi(z),0)$. The estimates \eqref{eq:Phi-bounds}--\eqref{eq:Phi-far} show that $i$ is a smooth embedding. For $h\in H_0$ define
\begin{equation}\label{eq:normal-frame}
\mathcal N_zh = \bigl((\Id-V_zV_z^*)h,V_z^*h\bigr).
\end{equation}
Since $V_z$ is an isometry,
\[
\norm{(\Id-V_zV_z^*)h}_2^2+\norm{V_z^*h}_2^2=\norm h_2^2,
\]
so $\mathcal N_z$ is an isometry. Its range is orthogonal to $i(z)$ and to $Di(z)(u^\perp)$, hence lies in the normal space to $i(S_{\ell_2^{m+1}}(R))$ in $S_{H_0\oplus\R^{m+1}}$. Both spaces have dimension $q$, so they coincide. Moreover,
\begin{equation}\label{eq:normal-lip}
\norm{\mathcal N_z-\mathcal N_w} \le B_1\norm{z-w}_2
\end{equation}
for an absolute constant $B_1$.

Define
\begin{equation}\label{eq:Theta}
\Theta(z,h) = \frac{i(z)+\mathcal N_zh}{\sqrt{1+\norm h_2^2}}.
\end{equation}
Because $\mathcal N_zh\perp i(z)$ and $\norm{\mathcal N_zh}_2=\norm h_2$, the image lies on the unit sphere. We choose an absolute $r\in(0,1/2]$ below and restrict $\Theta$ to $S_{\ell_2^{m+1}}(R)\times B_{H_0}(2r)$. After identifying $H_0$ isometrically with $\ell_2^q$, this has the domain and codomain in the statement.

We prove \eqref{eq:tube-inverse} when $s=\norm{z-w}_2\le s_0$. Let $u=z/R$ and put $v=P_u(w-z)$. Since $\norm z_2=\norm w_2=R$,
\[
\ip{u,w-z} = -\frac{s^2}{2R}, \qquad \norm v_2 = s\sqrt{1-\frac{s^2}{4R^2}} \ge s\sqrt{1-\frac{s^2}{4}}.
\]
Taylor's formula and \eqref{eq:Phi-bounds} give
\begin{equation}\label{eq:taylor-tangent}
i(w)-i(z)=Di(z)v+e, \qquad \norm e_2\le C_0s^2,
\end{equation}
where $C_0$ is absolute. Indeed, the component of $w-z$ in the $u$-direction has norm $s^2/(2R)$, and the Taylor remainder is bounded by an absolute multiple of $s^2$.

Put
\[
a(h)=(1+\norm h_2^2)^{-1/2}.
\]
Then
\begin{equation}\label{eq:tube-expansion}
\begin{split}
\Theta(w,k)-\Theta(z,h) ={}& a(h)\bigl(i(w)-i(z)+\mathcal N_z(k-h)\bigr)\\
&+ a(h)(\mathcal N_w-\mathcal N_z)k\\
&+ (a(k)-a(h))\bigl(i(w)+\mathcal N_wk\bigr).
\end{split}
\end{equation}
The vectors $Di(z)v$ and $\mathcal N_z(k-h)$ are orthogonal. Since $s_0\le1$ and $r\le1/2$, we have $\norm{Di(z)v}_2\ge s/4$, $\norm{\mathcal N_z(k-h)}_2=\norm{k-h}_2$, and $a(h)\ge1/\sqrt2$. Hence
\[
a(h)\norm{Di(z)v+\mathcal N_z(k-h)}_2 \ge \frac{s+\norm{k-h}_2}{8}.
\]
The scalar function $a$ satisfies
\[
\abs{a(k)-a(h)}\le2r\norm{k-h}_2,
\]
and \eqref{eq:normal-lip} bounds the variation of $\mathcal N_z$. Thus the remaining terms in \eqref{eq:tube-expansion}, including the error in \eqref{eq:taylor-tangent}, have total norm at most
\begin{equation}\label{eq:tube-error}
C_2\bigl(s^2+rs+r\norm{k-h}_2\bigr),
\end{equation}
where $C_2$ is absolute.

Choose $s_0$ sufficiently small and then $r$ sufficiently small so that the quantity in \eqref{eq:tube-error} is at most $(s+\norm{k-h}_2)/16$. We obtain
\begin{equation}\label{eq:tube-near}
s+\norm{k-h}_2 \le 16\norm{\Theta(z,h)-\Theta(w,k)}_2 \qquad (s\le s_0).
\end{equation}
As noted above, $s_0$ is chosen before $\eta$ in \eqref{eq:Phi-far}.

Now suppose $s\ge s_0$. From \eqref{eq:Theta},
\[
\norm{\Theta(z,h)-i(z)}_2\le2\norm h_2\le4r,
\]
and the same estimate holds for $(w,k)$. Hence \eqref{eq:Phi-far} gives
\[
\norm{\Theta(z,h)-\Theta(w,k)}_2\ge b_0-8r.
\]
Decrease $r$ so that $8r\le b_0/2$. Since $\norm{h-k}_2\le4r$ and $\min\{s,1\}\le1$, we obtain \eqref{eq:tube-inverse} also in this case, with
\[
D_0 = \max\left\{ 16,\frac{2(1+4r)}{b_0} \right\}.
\]
All constants are independent of $m$ and $R$.

The estimate \eqref{eq:tube-inverse} implies that $\Theta$ is injective on $S_{\ell_2^{m+1}}(R)\times B_{H_0}(2r)$. The local estimate \eqref{eq:tube-near} also implies that its differential is injective. Since the domain is compact, the restriction is a smooth embedding.

The manifold $S_{\ell_2^{m+1}}(R)\times\operatorname{int}B_{H_0}(2r)$ and the sphere $S_{H_0\oplus\R^{m+1}}$ have the same dimension. By invariance of domain,
\[
U= \Theta\bigl( S_{\ell_2^{m+1}}(R)\times \operatorname{int}B_{H_0}(2r) \bigr)
\]
is open in that sphere. The restriction of $\Theta$ is a homeomorphism onto $U$, so the boundary in $U$ of $\Theta(S_{\ell_2^{m+1}}(R)\times B_{H_0}(r))$ is
\[
\Theta\bigl( S_{\ell_2^{m+1}}(R)\times\partial B_{H_0}(r) \bigr).
\]
Since $\Theta(S_{\ell_2^{m+1}}(R)\times B_{H_0}(r))$ is compact and contained in $U$, its boundary in $U$ agrees with its boundary in $S_{H_0\oplus\R^{m+1}}$. This proves \textnormal{(ii)}.

Finally, let $D_h\Theta$ denote the derivative in the second variable. Differentiating \eqref{eq:Theta} at $h=0$ gives
\[
D_h\Theta(z,0)\xi=\mathcal N_z\xi.
\]
Under the fixed identification $H_0=\ell_2^q$, let $(e_j)_{j=1}^q$ be the standard orthonormal basis. Then
\[
\partial\Theta/\partial h_j(z,0)=\mathcal N_ze_j.
\]
Since $\mathcal N_z$ is an isometry onto the normal space, these vectors form an orthonormal basis of that space and depend smoothly on $z$. This proves \textnormal{(iii)}.
\end{proof}

We next construct a map that equals $p$ when $\norm h_2=r$ or $v_{a+1}\le\sqrt3/2$ and has relative degree of absolute value one. This map will be used with the parametrization in \Cref{lem:tube}.

\begin{lemma}\label{lem:collapse}
Let $r$ be the absolute constant in \Cref{lem:tube}. For integers $a,q\ge1$, put $p=(0,-1)\in S_{\ell_2^{a+q+1}}$. There is a continuous map
\[
\mathcal C: S_{\ell_2^{a+1}}\times B_{\ell_2^q}(r) \longrightarrow S_{\ell_2^{a+q+1}}
\]
with the following properties.

\begin{enumerate}[label=\textnormal{(\roman*)}]
\item
There is an absolute constant $L$ such that
\[
\norm{\mathcal C(v,h)-\mathcal C(\widetilde v,\widetilde h)}_2 \le L\bigl( \norm{v-\widetilde v}_2+\norm{h-\widetilde h}_2 \bigr)
\]
for all points in the domain.

\item
We have $\mathcal C(v,h)=p$ whenever $v_{a+1}\le\sqrt3/2$ or $\norm h_2=r$.

\item
The restriction of $\mathcal C$ to
\[
\bigl\{v\in S_{\ell_2^{a+1}}:v_{a+1}>\sqrt3/2\bigr\} \times\operatorname{int}B_{\ell_2^q}(r)
\]
is a homeomorphism onto $S_{\ell_2^{a+q+1}}\setminus\{p\}$. In particular, its relative degree has absolute value one.
\end{enumerate}
\end{lemma}

\begin{proof}
Write $v=(v',v_{a+1})$, with $v'\in\R^a$. If $v_{a+1}\ge\sqrt3/2$, then
$v=(v',\sqrt{1-\norm{v'}_2^2})$ and $\norm{v'}_2\le1/2$. For such $v$ and $h\in B_{\ell_2^q}(r)$, put
\[
w=(2v',h/r)\in B_{\ell_2^a}\times B_{\ell_2^q},
\qquad
\rho(w)=\max\{2\norm{v'}_2,\norm h_2/r\}.
\]
For $w\ne0$, define
\begin{equation}\label{eq:collapse}
\mathcal C(v,h)=
\left(
\sin(\pi\rho(w))\frac{w}{\norm w_2},
\cos(\pi\rho(w))
\right),
\end{equation}
and set $\mathcal C(v,h)=(0,1)$ when $w=0$. If $v_{a+1}<\sqrt3/2$, set $\mathcal C(v,h)=p$.

If $v_{a+1}=\sqrt3/2$, then $2\norm{v'}_2=1$, and if $\norm h_2=r$, then $\norm h_2/r=1$. Thus $\rho(w)=1$ in either case, so \eqref{eq:collapse} gives $\mathcal C(v,h)=p$. The two definitions therefore agree when $v_{a+1}=\sqrt3/2$. Since $\rho(w)\le\norm w_2$, continuity at $w=0$ is immediate. This proves continuity and \textnormal{(ii)}.

We prove \textnormal{(i)}. Write $w=(x,y)\in B_{\ell_2^a}\times B_{\ell_2^q}$. Then $\rho(w)=\max\{\norm x_2,\norm y_2\}$ is $1$-Lipschitz and $\rho(w)\le\norm w_2$. If $\norm w_2\ge\norm{\widetilde w}_2>0$, the radial estimate gives
\begin{align*}
&\left\|
\sin(\pi\rho(w))\frac{w}{\norm w_2}
-
\sin(\pi\rho(\widetilde w))
\frac{\widetilde w}{\norm{\widetilde w}_2}
\right\|_2\\
&\qquad\le
\pi\norm{w-\widetilde w}_2
+
\pi\norm{\widetilde w}_2
\frac{2\norm{w-\widetilde w}_2}{\norm w_2}
\le
3\pi\norm{w-\widetilde w}_2.
\end{align*}
The same estimate holds if one of the two vectors is zero. Since $w\mapsto\cos(\pi\rho(w))$ is $\pi$-Lipschitz, \eqref{eq:collapse} is $4\pi$-Lipschitz in $w$. Hence, whenever $v_{a+1},\widetilde v_{a+1}\ge\sqrt3/2$,
\[
\norm{\mathcal C(v,h)-\mathcal C(\widetilde v,\widetilde h)}_2
\le
8\pi\norm{v-\widetilde v}_2
+\frac{4\pi}{r}\norm{h-\widetilde h}_2.
\]

Suppose now that $v_{a+1}\ge\sqrt3/2>\widetilde v_{a+1}$. Let $b$ be the first point on a shortest spherical arc from $v$ to $\widetilde v$ such that $b_{a+1}=\sqrt3/2$. Then
$\norm{v-b}_2\le(\pi/2)\norm{v-\widetilde v}_2$, while
$\mathcal C(b,\widetilde h)=\mathcal C(\widetilde v,\widetilde h)=p$. Applying the preceding estimate to $(v,h)$ and $(b,\widetilde h)$ gives
\[
\norm{\mathcal C(v,h)-\mathcal C(\widetilde v,\widetilde h)}_2
\le
4\pi^2\norm{v-\widetilde v}_2
+\frac{4\pi}{r}\norm{h-\widetilde h}_2.
\]
If $v_{a+1},\widetilde v_{a+1}<\sqrt3/2$, both images equal $p$. Thus \textnormal{(i)} holds with
$L=\max\{4\pi^2,4\pi/r\}$.

We prove \textnormal{(iii)}. Let $(\zeta,t)\in S_{\ell_2^{a+q+1}}\setminus\{p\}$, with $\zeta\in\R^{a+q}$. Put $\rho=\pi^{-1}\arccos t\in[0,1)$. If $\rho=0$, then $(\zeta,t)=(0,1)$ and its unique preimage has $w=0$.

Assume $\rho>0$ and write
\[
\frac{\zeta}{\norm\zeta_2}=(\xi_1,\xi_2)\in\R^a\oplus\R^q.
\]
Equation \eqref{eq:collapse} forces
\[
w=
\frac{\rho}{\max\{\norm{\xi_1}_2,\norm{\xi_2}_2\}}
(\xi_1,\xi_2).
\]
If $w=(x,y)$, then
$v'=x/2$, $v_{a+1}=\sqrt{1-\norm{v'}_2^2}$, and $h=ry$. Since $\rho<1$, we have $\norm x_2<1$ and $\norm y_2<1$, so $v_{a+1}>\sqrt3/2$ and $\norm h_2<r$. Hence every point of $S_{\ell_2^{a+q+1}}\setminus\{p\}$ has a unique preimage in
\[
\bigl\{v\in S_{\ell_2^{a+1}}:v_{a+1}>\sqrt3/2\bigr\}
\times\operatorname{int}B_{\ell_2^q}(r).
\]
The inverse is continuous away from $(0,1)$. Since $\norm w_2\le\sqrt2\,\rho$, it is also continuous at $(0,1)$. Thus the restriction in \textnormal{(iii)} is a homeomorphism onto $S_{\ell_2^{a+q+1}}\setminus\{p\}$.

By \textnormal{(ii)}, the boundary is mapped to $p$. Therefore the relative degree may be computed at any point different from $p$. Such a point has exactly one preimage, with local degree $\pm1$, so the relative degree has absolute value one.

We finally record the smoothness needed below. If $v_{a+1}>\sqrt3/2$, $0<\norm{v'}_2<1/2$, and $h=0$, then $\rho=2\norm{v'}_2$ in a neighborhood of $(v,0)$. Hence \eqref{eq:collapse} is smooth there. Near the corresponding image,
$\norm{\xi_1}_2>\norm{\xi_2}_2$ and $\norm{\xi_1}_2>0$, so the inverse is smooth there as well. Thus $\mathcal C$ is a smooth local diffeomorphism at every such point.
\end{proof}

The preceding lemmas allow us to remove the radius $R$ from the Euclidean factor. We use the map $\Theta(z,h)$ from \Cref{lem:tube} to place $S_{\ell_2^{m+1}}(R)\times B_{\ell_2^q}(r)$ in a unit sphere, and then apply the map from \Cref{lem:collapse} to $(F(x,z),h)$.

\begin{lemma}\label{lem:radius}
There is an absolute constant $K<\infty$ with the following property. Let $E$ be a $d$-dimensional normed space, $d\ge2$, and suppose that
\[
F:S_E\times S_{\ell_2^{m+1}}(R) \longrightarrow S_{\ell_2^{d+m}}
\]
is continuous, has nonzero degree, and satisfies
\[
\norm{F(x,z)-F(y,z)}_2 \le \omega(\norm{x-y})
\]
and
\[
\norm{F(x,z)-F(x,w)}_2 \le \norm{z-w}_2
\]
for all $x,y\in S_E$ and $z,w\in S_{\ell_2^{m+1}}(R)$.

Then, for some $N\ge1$, there is a continuous map
\[
G:S_E\times S_{\ell_2^{N+1}} \longrightarrow S_{\ell_2^{d+N}}
\]
such that $|\deg G|=|\deg F|$ and
\begin{equation}\label{eq:G-estimates}
\begin{split}
\norm{G(x,s)-G(y,s)}_2 &\le K\omega(\norm{x-y}),\\
\norm{G(x,s)-G(x,t)}_2 &\le K\norm{s-t}_2.
\end{split}
\end{equation}
\end{lemma}

\begin{proof}
Apply \Cref{lem:tube} and let $q$ be the resulting integer. Put $N=m+q$ and
\[
\mathcal T= \Theta\bigl( S_{\ell_2^{m+1}}(R)\times B_{\ell_2^q}(r) \bigr) \subset S_{\ell_2^{N+1}}.
\]
Let $\mathcal C$ be the map in \Cref{lem:collapse}, with $a=d+m-1$. For $s=\Theta(z,h)\in\mathcal T$, define
\begin{equation}\label{eq:G-definition}
G(x,s)=\mathcal C(F(x,z),h),
\end{equation}
and set $G(x,s)=p$ outside $\mathcal T$.

This is continuous. Indeed, \Cref{lem:tube} identifies $\partial\mathcal T$ with the points for which $\norm h_2=r$, and \Cref{lem:collapse} sends all such points to $p$.

The estimate in the first variable follows immediately from the Lipschitz bound for $\mathcal C$:
\[
\norm{G(x,s)-G(y,s)}_2 \le L\,\omega(\norm{x-y}).
\]
To control the second variable, let $s=\Theta(z,h)$ and $t=\Theta(w,k)$ belong to $\mathcal T$. Then
\[
\norm{G(x,s)-G(x,t)}_2 \le L\bigl( \norm{F(x,z)-F(x,w)}_2+\norm{h-k}_2 \bigr).
\]
Since the image of $F$ lies on a unit sphere,
\[
\norm{F(x,z)-F(x,w)}_2\le 2\min\{1,\norm{z-w}_2\}.
\]
Hence \eqref{eq:tube-inverse} gives
\[
\norm{G(x,s)-G(x,t)}_2 \le 3LD_0\norm{s-t}_2.
\]
If one point lies outside $\mathcal T$, join the two points by a shortest spherical arc and stop at the first point of $\partial\mathcal T$. Since $G$ equals $p$ there and outside $\mathcal T$, the same estimate extends across the boundary, up to the absolute factor $\pi/2$. If both points lie outside $\mathcal T$, the images coincide. This proves \eqref{eq:G-estimates}.

We finally compare the degrees. Identify $S_E$ with a Euclidean sphere by an orientation-preserving radial homeomorphism and, for this calculation only, replace $F$ by a smooth map in the same homotopy class. The definition \eqref{eq:G-definition} extends the homotopy because $\mathcal C(F(x,z),h)=p$ when $\norm h_2=r$.

Choose a regular value $v$ of $F$ in the interior of the cap from \Cref{lem:collapse}, with $v'\ne0$. Then $\mathcal C(v,0)$ has under $G$ precisely the preimages $(x,\Theta(z,0))$ for which $F(x,z)=v$. By \Cref{lem:tube}, $(z,h)\mapsto\Theta(z,h)$ is a smooth coordinate map near every $(z,0)$, and \Cref{lem:collapse} shows that $\mathcal C$ is a local diffeomorphism at $(v,0)$. Thus the local degree of $G$ is the local degree of $F$, multiplied by the same orientation sign at every preimage. Summing the local degrees yields $\deg G=\pm\deg F$.

Homotopy invariance gives the same conclusion for the original continuous map $F$.
\end{proof}

\subsection{From products to spheres}\label{subsec:join}

For a normed space $E$ and a Hilbert space $H$, the direct sum $E\oplus_2H$ carries the norm
\[
\norm{(e,h)}_{E\oplus_2H} =\bigl(\norm e^2+\norm h_2^2\bigr)^{1/2}.
\]
We now pass from the product of unit spheres to the sphere of $E\oplus_2\ell_2^{N+1}$, preserving the absolute degree and a dimension-independent modulus of continuity.

\begin{lemma}\label{lem:join}
Let $E$ be a $d$-dimensional normed space, $d\ge2$, and let $G:S_E\times S_{\ell_2^{N+1}}\to S_{\ell_2^{d+N}}$ be continuous with nonzero degree. Suppose that
\[
\norm{G(u,v)-G(\widetilde u,\widetilde v)}_2 \le \omega_1(\norm{u-\widetilde u}) + B_1\norm{v-\widetilde v}_2
\]
for some modulus $\omega_1$ and some $B_1\ge0$. Then there is a continuous map
\[
Q:S_{E\oplus_2\ell_2^{N+1}} \longrightarrow S_{\ell_2^{d+N+1}}
\]
such that $|\deg Q|=|\deg G|$ and
\begin{equation}\label{eq:join-modulus}
\norm{Q(\xi)-Q(\eta)}_2 \le 5\pi\omega_1(\norm{\xi-\eta}) + (4\pi+2\pi B_1)\norm{\xi-\eta}
\end{equation}
for all $\xi,\eta\in S_{E\oplus_2\ell_2^{N+1}}$.
\end{lemma}

\begin{proof}
Let $(x,z)\in S_{E\oplus_2\ell_2^{N+1}}$. Thus $\norm{x}^2+\norm{z}_2^2=1.$
When $x$ and $z$ are both nonzero, define
\begin{equation}\label{eq:join}
Q(x,z)= \left( \sin(\pi\norm{x}^2) G\left(\frac{x}{\norm{x}},\frac{z}{\norm{z}_2}\right), \cos(\pi\norm{x}^2) \right).
\end{equation}
Set $Q(0,z)=(0,1)$ and $Q(x,0)=(0,-1)$. Since the sine factor tends to zero when either component tends to zero, $Q$ is continuous.

We compare the degrees. Write $(x,z)=(\sqrt t\,u,\sqrt{1-t}\,v)$, where $u\in S_E$, $v\in S_{\ell_2^{N+1}}$, and $0<t<1$. Then
\[
Q(x,z) = \bigl(\sin(\pi t)G(u,v),\cos(\pi t)\bigr).
\]
For this calculation, identify $S_E$ radially with $S^{d-1}$ and replace $G$, within its homotopy class, by a smooth map on $S^{d-1}\times S^N$. The same homotopy in \eqref{eq:join} extends continuously to the two ends $t=0,1$.

Choose a regular value $y\in S_{\ell_2^{d+N}}$ of $G$. The preimages of $(y,0)$ under $Q$ are exactly the points $(u/\sqrt2,v/\sqrt2)$ with $G(u,v)=y$. At each such point, the derivative of the last coordinate in the $t$-direction is $-\pi$, while the remaining coordinates restrict to $G$ at $t=1/2$. Hence all local degrees acquire the same additional orientation sign, and $\deg Q=\pm\deg G$.

We estimate the modulus of continuity. Let $\xi=(x,z)$ and $\eta=(y,w)$ have nonzero components, and set $\delta=\norm{\xi-\eta}$. Then $\abs{\norm{x}^2-\norm{y}^2}\le2\delta,$
while
\[
\left\| \frac{x}{\norm{x}}-\frac{y}{\norm{y}} \right\| \le \frac{2\delta}{\max\{\norm{x},\norm{y}\}}
\]
and
\[
\left\| \frac{z}{\norm{z}_2}-\frac{w}{\norm{w}_2} \right\|_2 \le \frac{2\delta}{\max\{\norm{z}_2,\norm{w}_2\}}.
\]
Put
\[
b=\min\{\sin(\pi\norm{x}^2),\sin(\pi\norm{y}^2)\}.
\]
Since $\sin(\pi t)\le\pi\min\{t,1-t\}$ on $[0,1]$,
\begin{equation}\label{eq:join-weight}
b \le \pi\min\{ \norm{x}^2,\norm{y}^2,\norm{z}_2^2,\norm{w}_2^2 \}.
\end{equation}
The change in the sine and cosine coefficients contributes at most $4\pi\delta$. Using the smaller sine coefficient for the difference of the two values of $G$, we get
\begin{align*}
\norm{Q(\xi)-Q(\eta)}_2 \le{}& 4\pi\delta\\
&+ b\left\| G\left(\frac{x}{\norm{x}},\frac{z}{\norm{z}_2}\right) - G\left(\frac{y}{\norm{y}},\frac{z}{\norm{z}_2}\right) \right\|_2\\
&+ bB_1 \left\| \frac{z}{\norm{z}_2}-\frac{w}{\norm{w}_2} \right\|_2.
\end{align*}
The last term is at most $2\pi B_1\delta$ by \eqref{eq:join-weight}.

For the first variable, apply the sphere-chain estimate \eqref{eq:sphere-chain} at scale $\delta$. We obtain
\begin{align*}
&b\left\| G\left(\frac{x}{\norm{x}},\frac{z}{\norm{z}_2}\right) - G\left(\frac{y}{\norm{y}},\frac{z}{\norm{z}_2}\right) \right\|_2\\
&\qquad\le b\left( \frac{4}{\max\{\norm{x},\norm{y}\}}+1 \right)\omega_1(\delta) \le 5\pi\omega_1(\delta).
\end{align*}
The last inequality follows from $b\le\pi\max\{\norm{x},\norm{y}\}^2$ and $\max\{\norm{x},\norm{y}\}\le1$. Combining the estimates gives \eqref{eq:join-modulus}. Points with a zero component follow by continuity.
\end{proof}

It remains to realize the added Hilbert directions inside $X$ while keeping the original space $E$ fixed. This follows from Dvoretzky's theorem applied inside a finite-codimensional subspace of $X$.

\begin{lemma}\label{lem:extension}
Let $E\subset X$ be finite dimensional and let $H$ be a finite-dimensional Hilbert space. Then there is a linear embedding $T:E\oplus_2 H\to X$ such that $T(e,0)=e$ for every $e\in E$ and
\begin{equation}\label{eq:extension}
\frac1{\sqrt{13}}\norm{(e,h)}_{E\oplus_2 H} \le \norm{T(e,h)} \le \sqrt5\,\norm{(e,h)}_{E\oplus_2 H}.
\end{equation}
\end{lemma}

\begin{proof}
If $E=\{0\}$, Dvoretzky's theorem gives an embedding of $H$ into $X$ with distortion at most two, which satisfies \eqref{eq:extension} after rescaling. Assume henceforth that $E\ne\{0\}$. By compactness of $S_E$ and Hahn--Banach, choose $\xi_1,\ldots,\xi_n\in S_{X^*}$ so that $\norm e\le2\max_i\abs{\xi_i(e)}$ for every $e\in E$. Their common kernel
\[
Y=\bigcap_{i=1}^n\ker\xi_i
\]
is infinite dimensional.

By Dvoretzky's theorem \cite{Dvoretzky}, $Y$ contains a copy of $H$ with distortion at most two. Thus there is a linear embedding $U:H\to Y$ such that
\[
\norm h_2\le\norm{Uh}\le2\norm h_2.
\]
Define $T(e,h)=e+Uh$.

Because $Uh\in Y$, the functionals $\xi_i$ see only the $E$-component: $\xi_i(e+Uh)=\xi_i(e)$. Hence $\norm e\le2\norm{e+Uh}$. Also
\[
\norm h_2\le\norm{Uh}\le\norm e+\norm{e+Uh} \le3\norm{e+Uh}.
\]
Therefore
\[
\bigl(\norm e^2+\norm h_2^2\bigr)^{1/2} \le \sqrt{13}\,\norm{e+Uh}.
\]
On the other hand,
\[
\norm{e+Uh} \le \norm e+2\norm h_2 \le \sqrt5\, \bigl(\norm e^2+\norm h_2^2\bigr)^{1/2}.
\]
This proves \eqref{eq:extension}.
\end{proof}

We combine the preceding lemmas to construct degree-one sphere maps on finite-dimensional subspaces containing a given subspace of $X$.

\begin{lemma}\label{lem:finite-enlargement}
Suppose that there are a constant $D\ge1$ and a modulus $\omega$ such that, for every finite-dimensional subspace $E\subset X$ with $\dim E\ge2$, there are an integer $m\ge1$ and a map $\varphi:S_E\to B_{\ell_2^m}$ satisfying
\[
\norm{\varphi(x)-\varphi(y)}_2 \le \omega(\norm{x-y})
\]
and
\[
\norm{x-y} \le D\norm{\varphi(x)-\varphi(y)}_2+\frac18 \qquad (x,y\in S_E).
\]
Then there is a modulus $\Omega$, depending only on $D$ and $\omega$, such that every finite-dimensional subspace $E_0\subset X$ is contained in a finite-dimensional subspace $V\subset X$, with $\dim V\ge2$, for which there is a continuous map
\[
f_V:S_V\longrightarrow S_{\ell_2^{\dim V}}
\]
of degree one satisfying
\[
\norm{f_V(x)-f_V(y)}_2 \le \Omega(\norm{x-y}) \qquad (x,y\in S_V).
\]
\end{lemma}

\begin{proof}
Choose a finite-dimensional subspace $E\subset X$ containing $E_0$, with $\dim E\ge2$. By \Cref{lem:product}, there are $R\ge1$ and a map of absolute degree one on $S_E\times S_{\ell_2^{m+1}}(R)$ whose first variable has modulus $32D\,\omega$ and whose second variable is $1$-Lipschitz. Applying \Cref{lem:radius}, we obtain an integer $N\ge1$ and a map
\[
G:S_E\times S_{\ell_2^{N+1}} \longrightarrow S_{\ell_2^{\dim E+N}}
\]
with $|\deg G|=1$ and
\[
\norm{G(u,v)-G(\widetilde u,\widetilde v)}_2 \le 32KD\,\omega(\norm{u-\widetilde u}) + K\norm{v-\widetilde v}_2,
\]
where $K$ is the absolute constant in \Cref{lem:radius}.

By \Cref{lem:join}, there is a map
\[
Q:S_{E\oplus_2\ell_2^{N+1}} \longrightarrow S_{\ell_2^{\dim E+N+1}}
\]
with $|\deg Q|=1$ and modulus bounded by
\[
t\longmapsto 160\pi KD\,\omega(t)+(4\pi+2\pi K)t.
\]
Thus the modulus of $Q$ depends only on $D$ and $\omega$.

By \Cref{lem:extension}, there is a linear embedding $T:E\oplus_2\ell_2^{N+1}\to X$ which fixes $E$ and satisfies
\[
\norm T\norm{T^{-1}}\le\sqrt{65}.
\]
Let $V=T(E\oplus_2\ell_2^{N+1}).$
Then $E_0\subset V$, and the radial map $\rho(v)=\frac{T^{-1}v}{\norm{T^{-1}v}}$
from $S_V$ onto $S_{E\oplus_2\ell_2^{N+1}}$ is Lipschitz with constant at most $2\sqrt{65}<18$.

Hence $Q\circ\rho:S_V\to S_{\ell_2^{\dim V}}$ has absolute degree one and a modulus depending only on $D$ and $\omega$. If its degree is $-1$, compose it with a reflection in one Euclidean coordinate. The resulting map $f_V$ has degree one and the same modulus.
\end{proof}

The proof preserves H\"older exponents. If
\[
\omega(t)=At^\alpha, \qquad 0<\alpha\le1,
\]
then $f_V$ may be chosen $\alpha$-H\"older, with a constant depending only on $A$ and $D$.

\subsection{Construction of the paving}\label{sec:main}

We now pass from the finite-dimensional maps of \Cref{lem:finite-enlargement} to a paving of $X$.

\begin{proof}[Proof of \Cref{prop:finite-paving}]
Let $\Phi:S_X\to H$ be a uniform embedding into a Hilbert space. Since $S_X$ is separable, the closed real linear span of $\Phi(S_X)$ is a separable Hilbert space and may be identified isometrically with a subspace of $\ell_2$. Composing with the Hilbert-sphere embedding described in \Cref{subsec:hilbertian}, we may therefore assume that $\Phi:S_X\to S_{\ell_2}$. By \Cref{lem:embedding-data}, there are a modulus $\omega$ and a constant $D\ge1$ such that every finite-dimensional subspace $E\subset X$, $\dim E\ge2$, admits a map $\varphi_E:S_E\to B_{\ell_2^{m_E}}$ satisfying the two estimates in \Cref{lem:embedding-data}. Hence \Cref{lem:finite-enlargement} applies.

Choose a dense sequence $(x_n)$ in $X$. Starting with $V_0=\{0\}$, apply \Cref{lem:finite-enlargement} successively to finite-dimensional subspaces containing $V_{n-1}$, $x_n$, and a vector outside $V_{n-1}$. We obtain
\[
V_1\subsetneq V_2\subsetneq\cdots\subset X
\]
and degree-one maps $f_n:S_{V_n}\longrightarrow S_{\ell_2^{\dim V_n}}$
with one common modulus of continuity.

Since $x_n\in V_n$ for every $n$, the union $\bigcup_nV_n$ is dense in $X$. Thus $(V_n)$ is a paving and the maps $f_n$ have degree one and one common modulus of continuity. The H\"older statement follows from the H\"older conclusion after \Cref{lem:finite-enlargement}.
\end{proof}

\section{Compatibility and proof of Theorem~A}\label{sec:compatibility}

The maps obtained in Section~\ref{sec:finite} need not agree on smaller spheres. The following proposition constructs one uniformly continuous map that preserves all their degrees.

\begin{proposition}[Compatibility]\label{prop:compatibility}
Let $(V_n)$ and $(H_n)$ be pavings of a separable infinite-dimensional real Banach space $X$ and a real Hilbert space $H$, respectively, with $\dim V_n=\dim H_n$. Suppose that maps $f_n:S_{V_n}\to S_{H_n}$ satisfy
\[
\norm{f_n(x)-f_n(y)}_2\le \omega(\norm{x-y})
\qquad(x,y\in S_{V_n},\ n\ge1),
\]
where $\omega$ is a modulus of continuity. Then there is a uniformly continuous map $g:S_X\to S_H$ such that
\[
g(S_{V_n})\subset S_{H_n},
\qquad
\deg(g|_{S_{V_n}})=\deg f_n
\qquad(n\ge1).
\]
The modulus of continuity of $g$ depends only on $\omega$.
\end{proposition}

We first regularize the finite-dimensional maps with a common $1/2$-H\"older constant and extend them to $S_X$. We then localize the extensions on pairwise disjoint caps. The final pasting preserves the degree on every member of the paving.

\subsection{Uniform regularization and extension}

For a Hilbert space \(H\), write
\[
\overline B_H(u,r)=\{v\in H:\norm{v-u}_2\le r\}.
\]

\begin{lemma}\label{lem:regularization}
Let \((M,d)\) be a nonempty metric space, let \(H\) be a finite-dimensional real Hilbert space, and let \(f:M\to S_H\) be continuous. Suppose that, for some \(\delta>0\),
\[
d(x,y)\le\delta \quad\Longrightarrow\quad \norm{f(x)-f(y)}_2\le\frac18.
\]
Then \(f\) is homotopic to a map \(q:M\to S_H\) such that
\[
\norm{q(x)-q(y)}_2\le \frac{4}{\sqrt{\delta}\,\beta^{3/2}}\sqrt{d(x,y)}
\qquad (x,y\in M).
\]
Here $\beta=(1-1/\sqrt{128})/2$.
\end{lemma}

\begin{proof}
For \(x\in M\), let \(P_x\) be the closed convex hull of $\{f(z):d(x,z)<\delta/2\},$
and let \(p_x\) be the point of \(P_x\) nearest to the origin. Since \(H\) is finite dimensional, \(P_x\) is compact.

If \(d(x,z),d(x,w)<\delta/2\), then \(d(z,w)<\delta\), and hence
\[
\ip{f(z),f(w)}
=1-\frac12\norm{f(z)-f(w)}_2^2
\ge\frac{127}{128}.
\]
By convexity and passage to the closure, \(\ip{u,v}\ge127/128\) for all \(u,v\in P_x\). In particular,
\[
\frac{127}{128}\le\norm{p_x}_2^2\le1.
\]
Since \(p_x\) is the metric projection of the origin onto \(P_x\),
\[
\ip{p_x,u-p_x}\ge0 \qquad (u\in P_x).
\]
Taking \(u=f(z)\) gives
\[
\norm{p_x-f(z)}_2^2
\le 1-\norm{p_x}_2^2
\le\frac1{128}
\qquad\bigl(d(x,z)<\delta/2\bigr).
\]

Now define
\[
C_x=\bigcap_{z\in M}
\overline B_H\left(f(z),\,1-\beta+\frac4\delta d(x,z)\right).
\]
We claim that \(p_x\) lies in every ball above with a margin \(\beta\):
\[
\norm{p_x-f(z)}_2
\le 1-2\beta+\frac4\delta d(x,z).
\]
Indeed, if \(d(x,z)<\delta/2\), this follows from
\[
\norm{p_x-f(z)}_2\le\frac1{\sqrt{128}}=1-2\beta.
\]
If \(d(x,z)\ge\delta/2\), then \(\norm{p_x-f(z)}_2\le2\), whereas
\[
1-2\beta+\frac4\delta d(x,z)\ge3-2\beta>2.
\]
Thus \(p_x\in C_x\).

Let \(u_x\) be the point of \(C_x\) nearest to the origin. Since \(p_x\in C_x\), while the ball indexed by \(z=x\) gives
\(\norm{u_x-f(x)}_2\le1-\beta\), we have
\begin{equation}\label{eq:u-bounds}
\beta\le\norm{u_x}_2\le1.
\end{equation}

We next compare \(u_x\) and \(u_y\). Put
\[
\theta=\frac{4d(x,y)}{\delta\beta+4d(x,y)}.
\]
For every \(z\in M\),
\[
\norm{u_x-f(z)}_2
\le1-\beta+\frac4\delta d(y,z)+\frac4\delta d(x,y),
\]
whereas
\[
\norm{p_y-f(z)}_2
\le1-2\beta+\frac4\delta d(y,z).
\]
Consequently,
\[
(1-\theta)u_x+\theta p_y\in C_y.
\]
Moreover,
\[
\norm{(1-\theta)u_x+\theta p_y-u_x}_2
\le2\theta
\le\frac{8}{\delta\beta}d(x,y).
\]
Since \(u_y\) is the metric projection of the origin onto \(C_y\),
\[
\ip{u_y,u_x-u_y}\ge-\frac{8}{\delta\beta}d(x,y).
\]
Interchanging \(x\) and \(y\) gives
\[
\ip{u_x,u_y-u_x}\ge-\frac{8}{\delta\beta}d(x,y).
\]
Adding the two inequalities yields
\begin{equation}\label{eq:u-holder}
\norm{u_x-u_y}_2^2
\le\frac{16}{\delta\beta}d(x,y).
\end{equation}

For nonzero vectors \(u,v\) in a Hilbert space,
\[
\norm{u-v}_2^2
=(\norm u_2-\norm v_2)^2
+\norm u_2\norm v_2
\norm*{\frac{u}{\norm u_2}-\frac{v}{\norm v_2}}_2^2,
\]
and therefore
\begin{equation}\label{eq:normalize}
\norm*{\frac{u}{\norm u_2}-\frac{v}{\norm v_2}}_2
\le
\frac{\norm{u-v}_2}{\min\{\norm u_2,\norm v_2\}}.
\end{equation}
Define \(q(x)=u_x/\norm{u_x}_2\). By
\eqref{eq:u-bounds}--\eqref{eq:normalize},
\[
\norm{q(x)-q(y)}_2
\le
\frac{4}{\sqrt{\delta}\,\beta^{3/2}}\sqrt{d(x,y)}.
\]

Finally,
\[
H_s(x)=
\frac{(1-s)f(x)+su_x}
{\norm{(1-s)f(x)+su_x}_2},
\qquad 0\le s\le1,
\]
is a homotopy from \(f\) to \(q\). Indeed,
\[
\norm{u_x-f(x)}_2\le1-\beta<1,
\]
so the numerator never vanishes.
\end{proof}

We shall use Minty's extension theorem in the \(1/2\)-H\"older case \cite{Minty}; see also \cite[Section~1.1]{NR}. Since only finite-dimensional Hilbert targets are needed here, we include the short proof.

\begin{lemma}\label{lem:holder-extension}
Let \(A\) be a nonempty subset of a metric space \((M,d)\), let \(H\) be a finite-dimensional real Hilbert space, and suppose that \(q:A\to H\) satisfies
\[
\norm{q(a)-q(b)}_2\le L\sqrt{d(a,b)}
\qquad(a,b\in A)
\]
for some \(L\ge0\). Then \(q\) extends to a map \(h:M\to H\) with
\[
\norm{h(x)-h(y)}_2\le L\sqrt{d(x,y)}
\qquad(x,y\in M).
\]
\end{lemma}

\begin{proof}
Fix \(x\in M\setminus A\). It is enough to show that
\[
\bigcap_{a\in A}
\overline B_H\bigl(q(a),L\sqrt{d(x,a)}\bigr)\ne\varnothing.
\]
Consider \(a_1,\ldots,a_n\in A\) and define
\[
\Phi(z)=
\max_{1\le i\le n}
\left\{
\norm{z-q(a_i)}_2^2-L^2d(x,a_i)
\right\},
\qquad z\in H.
\]
Since \(H\) is finite dimensional and \(\Phi(z)\to\infty\) as \(\norm z_2\to\infty\), the function \(\Phi\) attains its minimum at some \(z\in H\).

Let
\[
I=
\left\{
i:
\Phi(z)=
\norm{z-q(a_i)}_2^2-L^2d(x,a_i)
\right\}.
\]
Then
\[
z\in\operatorname{conv}\{q(a_i):i\in I\}.
\]
Indeed, otherwise a separating hyperplane provides a vector \(v\in H\) such that
\[
\ip{z-q(a_i),v}<0
\qquad(i\in I).
\]
For sufficiently small \(t>0\), every active term strictly decreases at \(z+tv\), while the remaining terms stay below the former maximum, contradicting the minimality of \(z\).

Choose \(p_i\ge0\), with \(p_i=0\) for \(i\notin I\), such that
\[
\sum_i p_i=1,
\qquad
z=\sum_i p_iq(a_i).
\]
Since all terms indexed by \(I\) equal \(\Phi(z)\),
\[
\begin{aligned}
\Phi(z)
&=\sum_i p_i
\left(
\norm{z-q(a_i)}_2^2-L^2d(x,a_i)
\right)\\
&=\frac12\sum_{i,j}p_ip_j
\norm{q(a_i)-q(a_j)}_2^2
-L^2\sum_i p_i d(x,a_i)\\
&\le
\frac{L^2}{2}\sum_{i,j}p_ip_jd(a_i,a_j)
-L^2\sum_i p_i d(x,a_i)
\le0.
\end{aligned}
\]
The last inequality follows from
\[
d(a_i,a_j)\le d(a_i,x)+d(x,a_j)
\]
and \(\sum_i p_i=1\). Hence
\[
\norm{z-q(a_i)}_2\le L\sqrt{d(x,a_i)}
\qquad(1\le i\le n),
\]
so every finite subfamily of the balls has nonempty intersection.

Fix \(a_0\in A\). The ball
\[
\overline B_H\bigl(q(a_0),L\sqrt{d(x,a_0)}\bigr)
\]
is compact, and its intersections with the remaining balls are closed and have the finite intersection property. Therefore
\[
\bigcap_{a\in A}
\overline B_H\bigl(q(a),L\sqrt{d(x,a)}\bigr)
\ne\varnothing.
\]
Thus \(q\) can be extended to \(A\cup\{x\}\) without increasing its \(1/2\)-H\"older constant.

Now consider all pairs \((D,h)\) such that \(A\subseteq D\subseteq M\), \(h:D\to H\) extends \(q\), and
\[
\norm{h(u)-h(v)}_2\le L\sqrt{d(u,v)}
\qquad(u,v\in D).
\]
Order these pairs by extension. The union of every chain is again an extension with the same estimate, so Zorn's lemma gives a maximal pair \((D,h)\). If \(D\ne M\), the one-point extension above enlarges \(D\), a contradiction. Hence \(D=M\).
\end{proof}

\begin{lemma}\label{lem:lift}
Let \(E\) be a nonzero finite-dimensional subspace of \(X\), let \(H_0\) be a finite-dimensional real Hilbert space, and suppose that \(q:S_E\to S_{H_0}\) satisfies
\[
\norm{q(x)-q(y)}_2\le L\sqrt{\norm{x-y}}
\qquad(x,y\in S_E).
\]
Let \(\xi\) be a unit vector orthogonal to \(H_0\). Then \(q\) extends to a map
\[
F:S_X\to S_{H_0\oplus_2\mathbb R\xi}
\]
such that
\[
\norm{F(x)-F(y)}_2\le2L\sqrt{\norm{x-y}}
\qquad(x,y\in S_X),
\]
and \(\ip{F(x),\xi}\ge0\) for every \(x\in S_X\).
\end{lemma}

\begin{proof}
By \Cref{lem:holder-extension}, \(q\) extends to a map \(h:S_X\to H_0\) with
\[
\norm{h(x)-h(y)}_2\le L\sqrt{\norm{x-y}}.
\]
Set \(a(x)=\dist(x,S_E)\) and
\[
U(x)=h(x)+L\sqrt{a(x)}\,\xi.
\]
For \(z\in S_E\),
\[
\norm{h(x)}_2
\ge1-\norm{h(x)-q(z)}_2
\ge1-L\sqrt{\norm{x-z}},
\]
and hence
\[
\norm{h(x)}_2\ge\bigl(1-L\sqrt{a(x)}\bigr)_+.
\]
Since \(H_0\perp\mathbb R\xi\),
\[
\norm{U(x)}_2^2
\ge
\bigl(1-L\sqrt{a(x)}\bigr)_+^2+L^2a(x)
\ge\frac12.
\]

The distance function \(a\) is \(1\)-Lipschitz, and
\[
|\sqrt{s}-\sqrt{t}|\le\sqrt{|s-t|}
\qquad(s,t\ge0).
\]
Therefore, by orthogonality,
\[
\begin{aligned}
\norm{U(x)-U(y)}_2^2
&=\norm{h(x)-h(y)}_2^2
 +L^2\bigl|\sqrt{a(x)}-\sqrt{a(y)}\bigr|^2\\
&\le2L^2\norm{x-y}.
\end{aligned}
\]
Define
\[
F(x)=\frac{U(x)}{\norm{U(x)}_2}.
\]
Using \eqref{eq:normalize} and \(\norm{U(x)}_2\ge2^{-1/2}\), we obtain
\[
\norm{F(x)-F(y)}_2
\le2L\sqrt{\norm{x-y}}.
\]
If \(x\in S_E\), then \(a(x)=0\), so \(U(x)=q(x)\) and \(F(x)=q(x)\). Finally,
\[
\ip{F(x),\xi}
=\frac{L\sqrt{a(x)}}{\norm{U(x)}_2}\ge0.
\]
\end{proof}

\subsection{Localization in disjoint caps}\label{subsec:caps}

We next construct a self-map of \(S_X\) which is constant off a fixed cap and remains homotopic to the identity on every finite-dimensional sphere containing the center of that cap.

\begin{lemma}\label{lem:cap}
Let \(V_0\subsetneq V\) be finite-dimensional subspaces of \(X\). Then there exist \(v\in S_V\) and a Lipschitz map \(\rho:S_X\to S_X\) such that
\[
\Lip(\rho)\le74880,
\qquad
\dist(v,V_0)=1,
\]
and
\[
\rho(x)=-v
\qquad
\text{whenever }\norm{x-v}\ge\frac16.
\]
Moreover, if \(E\subset X\) is finite dimensional and \(v\in E\), then \(\rho(S_E)\subset S_E\) and \(\rho|_{S_E}\) is homotopic to the identity on \(S_E\).
\end{lemma}

\begin{proof}
Choose \(\varphi\in S_{V^*}\) with \(\varphi|_{V_0}=0\), and choose \(v\in S_V\) such that \(\varphi(v)=1\). Extend \(\varphi\) to \(X\) by Hahn--Banach without changing its norm. For \(u\in V_0\),
\[
\norm{v-u}\ge |\varphi(v-u)|=1,
\]
and since \(0\in V_0\), we have \(\dist(v,V_0)=1\).

Define an equivalent norm on \(X\) by
\[
N(x)=|\varphi(x)|+\norm{x-\varphi(x)v}.
\]
Then
\[
\norm x\le N(x)\le3\norm x.
\]
Let \(S_N=\{x\in X:N(x)=1\}\). The radial maps
\[
A:S_X\to S_N,
\qquad
A(x)=\frac{x}{N(x)},
\]
and
\[
B:S_N\to S_X,
\qquad
B(z)=\frac{z}{\norm z},
\]
are both \(6\)-Lipschitz. Indeed,
\[
N(A(x)-A(x'))
\le
\frac{2N(x-x')}{\min\{N(x),N(x')\}}
\le6\norm{x-x'},
\]
while \(N(z)=1\) implies \(\norm z\ge1/3\), and hence
\[
\norm{B(z)-B(z')}
\le
\frac{2\norm{z-z'}}{\min\{\norm z,\norm{z'}\}}
\le6N(z-z').
\]

Every \(z\in S_N\) has a unique representation
\[
z=tv+y,
\qquad
t=\varphi(z),\quad
y\in\ker\varphi,\quad
|t|+\norm y=1.
\]
Set
\[
a(t)=\max\{-1,32t-31\},
\]
and
\[
b(t)=
\begin{cases}
0,&-1\le t\le15/16,\\[2mm]
\dfrac{2}{1-t}-32,&15/16<t<31/32,\\[2mm]
32,&31/32\le t\le1.
\end{cases}
\]
For \(|t|<1\),
\[
b(t)=\frac{1-|a(t)|}{1-|t|}.
\]
Hence the map
\[
Q(tv+y)=a(t)v+b(t)y,
\qquad
Q(\pm v)=\pm v,
\]
takes \(S_N\) into itself.

The functions \(a\) and \(b\) satisfy
\[
\Lip(a)\le32,
\qquad
0\le b\le32,
\qquad
\Lip(b)\le2048.
\]
For \(z=tv+y\) and \(z'=t'v+y'\),
\[
\begin{aligned}
N(Q(z)-Q(z'))
&\le |a(t)-a(t')|
   +\norm{b(t)y-b(t')y'}\\
&\le2080|t-t'|+32\norm{y-y'}\\
&\le2080N(z-z').
\end{aligned}
\]
Thus \(Q\) is \(2080\)-Lipschitz. Define
\[
\rho=B\circ Q\circ A.
\]
It follows that
\[
\Lip(\rho)\le6\cdot2080\cdot6=74880.
\]

We now locate the set on which \(\rho\) is not equal to \(-v\). Let \(z=tv+y\in S_N\) with \(t>0\). Since \(\norm y=1-t\), we have \(\norm z\ge t\) and
\[
|\norm z-t|\le\norm y=1-t.
\]
Therefore
\begin{equation}\label{eq:cap-radius}
\norm{B(z)-v}
\le
\frac{\norm y+|\norm z-t|}{\norm z}
\le
\frac{2(1-t)}{t}.
\end{equation}
If \(Q(z)\ne-v\), then necessarily \(t>15/16\), and hence
\[
\norm{B(z)-v}<\frac{2}{15}<\frac16.
\]
Applying this to \(z=A(x)\), and using \(B(A(x))=x\), gives
\[
\rho(x)\ne-v
\quad\Longrightarrow\quad
\norm{x-v}<\frac16.
\]
Thus \(\rho(x)=-v\) whenever \(\norm{x-v}\ge1/6\).

It remains to verify the homotopy statement. For \(0\le s\le1\), define
\[
a_s(t)=(1-s)t+sa(t),
\]
and, for \(|t|<1\),
\[
b_s(t)=\frac{1-|a_s(t)|}{1-|t|}.
\]
Set
\[
Q_s(tv+y)=a_s(t)v+b_s(t)y,
\qquad
Q_s(\pm v)=\pm v.
\]
Each \(a_s\) maps \([-1,1]\) into itself, fixes the endpoints, and has Lipschitz constant at most \(32\). Hence
\[
1-|a_s(t)|\le32(1-|t|),
\]
so \(0\le b_s(t)\le32\). Since \(\norm y=1-|t|\), the term \(b_s(t)y\) tends to \(0\), uniformly in \(s\), as \(t\to\pm1\). Therefore \((Q_s)_{0\le s\le1}\) is a continuous homotopy from the identity on \(S_N\) to \(Q\).

Finally, suppose that \(E\subset X\) is finite dimensional and \(v\in E\). If \(z\in E\), then \(z-\varphi(z)v\in E\), so every \(Q_s\) leaves \(S_N\cap E\) invariant. The radial maps \(A\) and \(B\) also preserve \(E\). Consequently,
\[
B\circ Q_s\circ A
\]
is a homotopy on \(S_E\) from the identity to \(\rho|_{S_E}\).
\end{proof}

\subsection{Pasting and degree computation}

\begin{proof}[Proof of \Cref{prop:compatibility}]
Work with the given nested paving $(H_n)$ of $H$. Choose $\delta\in(0,1]$ with $\omega(\delta)\le1/8$. By \Cref{lem:regularization}, each $f_n$ is homotopic to a map $q_n:S_{V_n}\to S_{H_n}$ with
\[
\norm{q_n(x)-q_n(y)}_2\le L\sqrt{\norm{x-y}}
\qquad(x,y\in S_{V_n}),
\]
where $L$ depends only on $\omega$. Hence $\deg q_n=\deg f_n$.

Since $\dim H_{n+1}>\dim H_n$, choose a unit vector $\xi_n\in H_{n+1}\cap H_n^\perp$. By \Cref{lem:lift}, $q_n$ extends to a map $F_n:S_X\to S_{H_{n+1}}$ such that
\begin{equation}\label{eq:lift-properties}
F_n|_{S_{V_n}}=q_n,\qquad
\norm{F_n(x)-F_n(y)}_2\le2L\sqrt{\norm{x-y}},
\qquad
\ip{F_n(x),\xi_n}\ge0
\end{equation}
for all $x,y\in S_X$.

Set $V_0=\{0\}$. Apply \Cref{lem:cap} to $V_{n-1}\subsetneq V_n$, and let $v_n$, $\rho_n$, and $\varphi_n$ be the resulting vector, map, and norming functional. Put $D=74880$. If $j<n$, then $v_j\in V_{n-1}$, while $\dist(v_n,V_{n-1})=1$, so $\norm{v_n-v_j}\ge1$. Hence the sets
\[
C_n=\{x\in S_X:\norm{x-v_n}<1/6\}
\]
are pairwise disjoint. Moreover, $C_n\cap S_{V_{n-1}}=\varnothing$, and $\rho_n=-v_n$ on $S_X\setminus C_n$.

Fix $p\in S_{H_1}$. Since $\dim H_n\ge2$, choose an orientation-preserving orthogonal transformation $R_n:H_n\to H_n$ such that $R_nq_n(-v_n)=p$, and extend $R_n$ to $H_{n+1}$ by the identity on $H_{n+1}\cap H_n^\perp$. Define
\[
G_n=R_n\circ F_n\circ\rho_n:S_X\longrightarrow S_{H_{n+1}}.
\]
Since $\rho_n=-v_n$ outside $C_n$, we have $G_n=p$ there. Also,
\[
\norm{G_n(x)-G_n(y)}_2
\le2L\sqrt D\,\sqrt{\norm{x-y}}
\qquad(x,y\in S_X).
\]

Define
\begin{equation}\label{eq:paste}
g(x)=
\begin{cases}
G_n(x),&x\in C_n\text{ for some }n,\\
p,&x\notin\bigcup_n C_n.
\end{cases}
\end{equation}
Since the $C_n$ are pairwise disjoint, $g$ is well defined. If $x,y\in C_n$, the preceding estimate applies. If $x\in C_n$ and $y\in C_j$ with $n\ne j$, then $G_n(y)=G_j(x)=p$, and therefore
\[
\begin{aligned}
\norm{g(x)-g(y)}_2
&\le
\norm{G_n(x)-G_n(y)}_2+
\norm{G_j(x)-G_j(y)}_2\\
&\le4L\sqrt D\,\sqrt{\norm{x-y}}.
\end{aligned}
\]
If only one of $x,y$ lies in some $C_n$, the same estimate uses one term, and if neither does, then $g(x)=g(y)=p$. Thus $g$ is uniformly continuous, with a modulus depending only on $\omega$.

Fix $m\ge1$. If $n>m$, then $C_n\cap S_{V_m}=\varnothing$. If $n<m$, then
$G_n(S_{V_m})\subset S_{H_{n+1}}\subset S_{H_m}$. For $n=m$, \Cref{lem:cap} gives $\rho_m(S_{V_m})\subset S_{V_m}$, and \eqref{eq:lift-properties} gives $F_m=q_m$ on $S_{V_m}$. Hence
\[
g(S_{V_m})\subset S_{H_m}.
\]

We now compute $\deg(g|_{S_{V_m}})$. For $n\le m$, recall from \Cref{lem:cap} that
\[
N_n(x)=|\varphi_n(x)|+\norm{x-\varphi_n(x)v_n},\qquad
A_n(x)=\frac{x}{N_n(x)},\qquad
B_n(z)=\frac{z}{\norm z}.
\]
Let $S_{N_n}=\{x:N_n(x)=1\}$, so that $A_n:S_X\to S_{N_n}$ and $B_n:S_{N_n}\to S_X$. Define
\[
D_n^m=
\left\{
B_n\bigl((1-\norm y)v_n+y\bigr):
y\in V_m\cap\ker\varphi_n,\ \norm y<1/15
\right\}.
\]
This is an open ball in $S_{V_m}$, and its closure is obtained by allowing $\norm y\le1/15$. By \eqref{eq:cap-radius}, the corresponding coordinate satisfies $t\ge14/15$, and therefore
\[
\norm{x-v_n}\le\frac17
\qquad(x\in\overline D_n^m).
\]
Hence $\overline D_n^m\subset C_n$, so these closed balls are pairwise disjoint.

If $G_n(x)\ne p$, then $\rho_n(x)\ne-v_n$. In the notation of \Cref{lem:cap}, this requires the coordinate $t$ of $A_n(x)$ to satisfy $t>15/16$. Since $15/16>14/15$, the map $G_n|_{S_{V_m}}$ equals $p$ outside a compact subset of $D_n^m$.

Let $\pi:S_{V_m}\to Q_m$ map all points outside $\bigcup_{n=1}^m D_n^m$ to one point and keep distinct points of each $D_n^m$ distinct. Then
\[
Q_m=\bigvee_{n=1}^m \overline D_n^m/\partial D_n^m.
\]
Give each sphere $\overline D_n^m/\partial D_n^m$ the orientation induced from $S_{V_m}$. The map $\pi$ sends the fundamental class of $S_{V_m}$ to the sum of the fundamental classes of these spheres. Since $g=p$ outside the $D_n^m$, there is a continuous map $\widetilde g:Q_m\to S_{H_m}$ with $g|_{S_{V_m}}=\widetilde g\circ\pi$. On $\overline D_n^m/\partial D_n^m$, the map $\widetilde g$ is induced by $G_n|_{S_{V_m}}$. Therefore
\begin{equation}\label{eq:add-degree}
\deg(g|_{S_{V_m}})
=
\sum_{n=1}^m\deg(G_n|_{S_{V_m}}).
\end{equation}

For $n=m$, the restriction $\rho_m|_{S_{V_m}}$ is homotopic to the identity, $F_m=q_m$ on $S_{V_m}$, and $R_m$ preserves orientation. Hence
\[
\deg(G_m|_{S_{V_m}})
=
\deg q_m
=
\deg f_m.
\]
If $n<m$, then $\xi_n\in H_m$. Since $R_n\xi_n=\xi_n$, \eqref{eq:lift-properties} gives
\[
G_n(S_{V_m})
\subset
\{u\in S_{H_m}:\ip{u,\xi_n}\ge0\}.
\]
Thus $G_n|_{S_{V_m}}$ misses $-\xi_n$, is null-homotopic as a map into $S_{H_m}$, and has degree zero. Substituting these values into \eqref{eq:add-degree} gives
\[
\deg(g|_{S_{V_m}})=\deg f_m.
\]
This proves the proposition.
\end{proof}

The construction also gives
\[
\norm{g(x)-g(y)}_2\le C\norm{x-y}^{1/2} \qquad(x,y\in S_X),
\]
where $C$ depends only on $\omega$. Thus the compatible maps can be chosen with a common $1/2$-H\"older constant.

\subsection{Completion of the proof}

\begin{proof}[Proof of Theorem~A]
By \Cref{prop:finite-paving}, there are a paving $(V_n)$ of $X$ and degree-one maps
\[
f_n:S_{V_n}\longrightarrow S_{\ell_2^{\dim V_n}}
\]
with one common modulus of continuity. Identify the target spaces by orientation-preserving linear isometries with the nested coordinate subspaces
\[
H_n=\operatorname{span}\{e_1,\ldots,e_{\dim V_n}\}\subset\ell_2.
\]
Then $(H_n)$ is a paving of $\ell_2$. Applying \Cref{prop:compatibility} gives a uniformly continuous map $g:S_X\to S_{\ell_2}$ such that $g(S_{V_n})\subset S_{H_n}$ and
\[
\deg(g|_{S_{V_n}})=\deg f_n=1
\qquad(n\ge1).
\]
Restricting $g$ to $S_{\bigcup_nV_n}$ gives the map required in \Cref{def:KY}. Hence $X$ has Property~(H).
\end{proof}

\section*{Acknowledgements}

The authors are particularly grateful to Professor Guoliang Yu for their valuable discussions, guidance, and encouragement throughout the development of this work on Property~$(H)$. The work was supported by the National Natural Science
Foundation of China (Grant No.~12471132). During the preparation of this work, the authors used GPT-5.6 Sol to
assist with language editing and refinement of mathematical arguments. The authors take full responsibility for the content of this paper.


\begin{thebibliography}{99}

\bibitem{BL2000}Y. Benyamini and J. Lindenstrauss,
\emph{Geometric Nonlinear Functional Analysis}, Vol.~1, AMS Colloquium Publications, Vol.~48, American Mathematical Society, Providence, RI, 2000.

\bibitem{CW}Q. Cheng and Q. Wang, \emph{On Banach spaces with Kasparov and Yu's Property (H)}, J. Math. Anal. Appl. \textbf{457} (2018), 200--213. \href{https://doi.org/10.1016/j.jmaa.2017.08.005}{doi:10.1016/j.jmaa.2017.08.005}.

\bibitem{Dvoretzky}A.~Dvoretzky, \emph{Some results on convex bodies and Banach spaces},
Proc. Internat. Sympos. Linear Spaces (Jerusalem, 1960), Jerusalem Academic Press, Jerusalem; Pergamon Press, Oxford, 1961, pp.~123--160.

\bibitem{KY}G.~Kasparov and G.~Yu, \emph{The Novikov conjecture and geometry of Banach
spaces}, Geom. Topol. \textbf{16} (2012), no.~3, 1859--1880.

\bibitem{Minty}G. J. Minty, \emph{On the extension of Lipschitz, Lipschitz-H\"older continuous, and monotone functions}, Bull. Amer. Math. Soc. \textbf{76} (1970), 334--339.

\bibitem{NR}A. Naor and Y. Rabani, \emph{On Lipschitz extension from finite subsets}, Israel J. Math. \textbf{219} (2017), 115--161. \href{https://doi.org/10.1007/s11856-017-1475-1}{doi:10.1007/s11856-017-1475-1}.

\bibitem{Ricard}\'E.~Ricard, \emph{H\"older estimates for the noncommutative Mazur maps},
Arch. Math. (Basel) \textbf{104} (2015), no.~1, 37--45. \href{https://doi.org/10.1007/s00013-014-0710-9}{doi:10.1007/s00013-014-0710-9}. \href{https://arxiv.org/abs/1407.8334}{arXiv:1407.8334}.

\bibitem{Schaffer1967}J.~J.~Sch\"affer, \emph{Inner diameter, perimeter, and girth of spheres},
Math. Ann. \textbf{173} (1967), 59--79; addendum, 79--82.

\end{thebibliography}
\end{document}